\documentclass[11pt]{amsart}

\usepackage{amssymb, amsthm, amsmath,euscript}

\newtheorem{theorem}{Theorem}[section]

\newtheorem{lemma}[theorem]{Lemma}

\newtheorem*{theorem*}{Theorem}{\bf}{\it}
\newtheorem*{proposition*}{Proposition}{\bf}{\it}
\newtheorem*{observation*}{Observation}{\bf}{\it}
\newtheorem*{lemma*}{Lemma}{\bf}{\it}

\theoremstyle{definition}

\theoremstyle{remark}
\newtheorem{remark}[theorem]{Remark}
\newcommand{\sym}{axially-symmetric }
\newcommand{\F}{\EuScript F}
\newcommand{\HH}{\EuScript H}

\newcommand{\R}{\mathbb R}
\newcommand{\C}{\EuScript C}

\begin{document}
\title{On the  higher-dimensional harmonic analog of the Levinson log log theorem.}
\author{Alexander Logunov}

\begin{abstract}
 Let $M\colon (0,1) \to [e,+\infty)$ be a decreasing function such that $\int\limits_{0}^{1}\log\log M(y)dy<+\infty$. Consider the set $H_M$ of all functions $u$ harmonic in $P:=\{
(x,y)\in \R^n:  x\in \R^{n-1}, y\in \R, |x|<1, |y|<1 \}$  and satisfying $|u(x,y)| \leq M(|y|)$. We prove that $H_M$ is a normal family in $P$.

\end{abstract}

\maketitle
\section{Preliminaries.}


 Let $P$ be a rectangle $(-a,a)\times(-b,b)$ in $\mathbb{R}^2$ and let $M:(0,b)\to [e,+\infty)$ be a decreasing function. Consider the set $\EuScript{F}_M$ of all functions $f$ holomorhpic in $P$ such that $|f(x,y)| \leq M(|y|)$, $(x,y)\in P$. The classical Levinson theorem asserts that $\F_M$ is a normal family in $P$ if $\int\limits_{0}^{b}\log\log M(y)dy<+\infty$. We refer the reader to \cite{B71, C, D96,D58,D88,G70,H71,K88,L40,M60,M76,R91,R09,R78,Y81} for various proofs, history of the question and related topics. This statement is sharp, i.e. for regular (continuous and decreasing) majorants $M$ the family $\F_M$ is normal if and only if $\int\limits_{0}^{b}\log\log M(y)dy<+\infty$ (see \cite{K88},p.379--383 and \cite{B71}).

 A function $\log^+x$ is defined by $$\log^+x= \begin{cases} \log x, \quad x\geq1
\\ 0, \quad  x\leq 1.
 \end{cases}$$
 Our result is the following theorem, which extends the Levinson $\log\log$ theorem for holomorphic functions to  harmonic  functions in $\mathbb{R}^n$, $n \geq 2$.  
\begin{theorem} \label{main}
Let $\Omega$ denote the set $\{(x,y):  x\in \R^{n-1}, y\in \R,|x|<R, |y|<H \}$, where $R$ and $H$ are some positive numbers. Suppose a function $M\!\!:(0,H) \to  \R_+$ is  decreasing and
\begin{equation} \label{loglog}
\int\limits_{0}^{H}\log^{+}\log^{+}M(y)dy<+\infty.
\end{equation}
 Then the set $\HH_M$ of all functions $u$ harmonic in $\Omega$ and satisfying $|u(x,y)| \leq M(|y|)$, $(x,y) \in \Omega$, is uniformly bounded on any compact subset of $\Omega$.
\end{theorem}
  This result has been proved by Dyn'kin in \cite{D96} by a different method under some stronger regularity conditions imposed on $M$. For any compact set $K\subset \Omega$ our approach provides an explicit estimate for $\sup\limits_{u\in \HH_M}\sup\limits_{K}|u|$ in terms of $M$, $K$ and $\Omega$. We obtain Theorem \ref{main} as a corollary of the "holomorphic" Levinson theorem by a reduction to axially-symmetric functions $u$. First, we prove the Theorem \ref{main} in dimension $4$, it implies the $3$-dimensional case. Then we reduce the case of odd $n$ to the case $n=3$. The case of even $n$ follows by adding the fake variable. The main obstacle, which appears in the higher-dimensional harmonic analog of the Levinson $\log\log$ theorem, is the fact that $\log|\nabla u|$ is not necesseraily subharmonic for a general harmonic function $u$ in $\mathbb{R}^n$ if $n\geq 3$.  

   Some of the proofs of the "holomorphic" Levinson $\log\log$ theorem are of complex nature, some use implicitly or explicitly  harmonic measure estimates in cusp-like domains, but most of the proofs require the monotinicity condition on $M$, except for  the brilliant idea due to Domar (see  \cite{K88},\cite{D58},\cite{D88}), which avoids any regularity assumptions on $M$, even the monotonicity. We will sketch Domar's proof in Section \ref{D}, and use it to obtain explicit uniform estimates for $\HH_M$ in higher dimensions. We don't know whether Theorem \ref{main} is valid for arbitrary majorants $M$ satisfying \eqref{loglog} (even for $n=2$). 

  For any $x,y \in \R^n$ let $d(x,y)$ denote the Euclidean distance between $x$ and $y$. For any $X,Y \subset \R^n$ we use the notation $d(X,Y)$ for $\inf\{d(x,y): x\in X, y \in Y  \}$. The symbol $\lambda_n$ will denote the $n$-dimensional Lebesgue measure in $\R^n$.

\section{Domar's argument} \label{D}
 
\begin{theorem} \label{Levinson}
 Let $f$ be a holomorphic function in a rectangle $P:={(-a,a)\times(-b, b)}$. Suppose that a function $M(y)$ satisfies $\int\limits_{-b}^{b}\log^{+}\log^{+}M(y)dy<+\infty$ and   $|f(x+iy)|\leq M(y)$ for  all $(x,y) \in P$.
Then for any compact set $K\subset P$ there exists a constant $C=C(M,d(K,\partial P))$   such that $\sup\limits_{K}|f|<C$.
 \end{theorem}

     Theorem \ref{Levinson} immediately follows from the next lemma on subharmonic functions, since $\log |f|$ is subharmonic.
\begin{lemma} \label{subharmonic lemma}
 Let $v$ be a subharmonic function in a rectangle $P:={(-a,a)\times(-b, b)}$. Suppose that a function $\tilde M$ satisfies $\int\limits_{-b}^{b}\log^{+}\tilde M(y)dy<+\infty$ and   $v(x+iy)\leq \tilde M(y)$ for  all $(x,y) \in P$.
 Then for any compact set $K\subset P$ there exists a constant $C=C(\tilde M,d(K,\partial \Omega))$   such that $\sup\limits_{K}v\leq C$.
\end{lemma}
 \begin{proof}[Sketch of the proof]

  Let $F(t):= \lambda_1(\{y \in (-b,b)\colon \tilde M(y)\geq t\})$ denote the complementary cumulative distribution function of $\tilde M(y)$. The logarithmic integral condition $\int_{-b}^{b}\log^{+}\tilde M(y)dy<+\infty$ can be reformulated in terms of $F$, namely $\sum\limits_{i=0}^{+\infty} F(2^i)<+\infty$ if
 $\int_{-b}^{b}\log^{+}\tilde M(y)dy<+\infty$ (see \cite{K88}, p.378--379). Then there exists a positive number $C$ such that 
\vspace{-5pt}
\begin{equation}
\label{DE} \sum\limits_{i=-1}^{+\infty} F( 2^i C)<\frac{\pi}{8} d(K,\partial P). 
\end{equation}

Our aim is to show that $\sup\limits_{K} v \leq C$. Assume the contrary. Suppose there is $z_0\in K$ with $v(z_0)>C$.
 Let $A_t$ denote the set $\{z \in P: u(z) \geq t \}$. 

\begin{proposition*}
 If a point $z \in P$ satisfies $v(z) \geq \C$ with $\C>0$,  and $d(z, \partial P) > \frac{8}{\pi}F(\C/2)$, then there is a $\zeta \in P$ such that $d(z,\zeta)\leq \frac{8}{\pi}F(\C/2)$ and $v(\zeta) \geq 2\C$.
\end{proposition*}
 Consider the ball $B$ centered at $z$ with radius $r=\frac{8}{\pi}F(\C/2)$, note that $B\in P$, since $d(z, \partial P) > \frac{8}{\pi}F(\C/2)$. Now, the subharmonicity of $v$ will be exploited:
$\C\leq v(z) \leq \frac{1}{\lambda_2(B)} \int\limits_{B} v = \frac{1}{\lambda_2(B)} ( \int\limits_{B \setminus A_{\C/2}} v + \int\limits_{B \cap A_{\C/2}} v) \leq
\C/2 +  \frac{1}{\lambda_2(B)} \int\limits_{B \cap A_{\C/2}} v$.
 Hence $\C/2 \leq \frac{1}{\lambda_2(B)} \int\limits_{B \cap A_{\C/2}} v \leq \frac{1}{\pi r^2}  \sup\limits_{B} v  \cdot \lambda_2(B \cap A_{\C/2}) \leq \frac{1}{\pi r^2}  \sup\limits_{B} v  \cdot \lambda_1(\{x| \quad \!\!\! \exists\quad \!\!\!  y:(x,y) \in B \cap A_{\C/2}\}) \cdot \lambda_1(\{y: \quad \!\!\!\exists \!\!\! \quad x: (x,y) \in B \cap A_{\C/2}\}) \leq \frac{1}{\pi r^2}  \sup\limits_{B} v  \cdot 2r F(\C/2)= \frac{1}{4} \sup\limits_{B} v $. 
Thus $2C \leq \sup\limits_{B} v $ and the proposition is proved.


Using the proposition and taking $z_0$ in place of $z$ and $C$ in place of $\C$ we obtain a point $z_1$ such that $v(z_1) \geq 2C$ and $d(z_1,z_0)\leq \frac{8}{\pi}F(C/2)$.  Recall that $d(z_0, \partial P) >  \frac{8}{\pi} \sum\limits_{i=-1}^{+\infty} F( 2^i C)$, hence $d(z_1,\partial P) > \frac{8}{\pi} \sum\limits_{i=0}^{+\infty} F( 2^i C)$.  
 Exploiting the proposition infinitely many times we obtain a sequence $\{ z_i \}_{i=0}^{+\infty}$ such that $v(z_i) \geq 2^i C$ and $d(z_i, z_{i+1}) \leq \frac{8}{\pi}F(2^{i-1}C)$.
 By \eqref{DE} $z_i$ have a limit point $z\in P$,  hence $v(z)\geq \lim\limits_{i \to \infty} v(z_i)=+\infty$, and  a contradiction is obtained.
\end{proof}
 \begin{remark} \label{estimateD}
 Domar's argument also provides explicit estimates in Threorem \ref{Levinson} of $C(M,d(K,\partial P))$. Define $F(t)$  by $F(t)=\lambda_1(\{ y : \log^+M(y)\geq t\})$. If $C>0$ and $d(K, \partial P) >  \frac{8}{\pi} \sum\limits_{i=-1}^{+\infty} F( 2^i C)$ , then $|f| \leq \exp(C)$ on $K$.
\end{remark}


\section{Axially-Symmetric Harmonic Functions.}

   Consider $\R^n=\{x= (x_1,\dots, x_n): x_i \in \R \}$.
 By $\rho$ we denote $\sqrt{\sum\limits_{i=1}^{n-1} x_i^2}$ and $h:=x_n$.  A function $u$ defined in $\R^n$ is called \sym if $u=u(\rho,h)$, i.e. $u$ is invariant under orthogonal transforamtions of the first $(n-1)$ coordinates.   An \sym harmonic function $u$ satisfies the elliptic Euler-Darboux equation:
\begin{equation} \label{ED}
 \frac{\partial^2 u}{\partial \rho^2}  + \frac{\partial^2 u}{\partial h ^2}  + \frac{n-2}{\rho}\frac{\partial u}{\partial \rho}=0. 
\end{equation} 
 We are going to use two ideas. The first one reduces \sym harmonic functions in $\R^4$ to ordinary harmonic functions in $\R^2$. The second trick reduces \sym harmonic functions in $\R^{2k+3}$ to harmonic functions in $\R^3$. It will help in dimension $n \geq 5$.  We refer the reader to \cite{A87}, \cite{H96}, \cite{E56}, \cite{E65}, \cite{K}, \cite{H91}, \cite{R68}, \cite{R74}, \cite{W48}, \cite{W53} and references therein, where these and related ideas  appear in a different context, however we are not able to locate their origin. 

\subsection{From $\mathbb{R}^4$ to $\mathbb{R}^2$  \label{from R4} } 

  Suppose $u$ is an \sym harmonic function in an \sym domain $\Omega \subset \mathbb{R}^4$. Consider the set $\tilde \Omega_+ \subset \mathbb{R}^2 $ defined by  $x \in \Omega \iff (\rho(x), h (x)) \in \tilde \Omega_+ . $
 It is easy to see from \eqref{ED} that the function 
\begin{equation} \label{R4}
\tilde u(\rho,h) = \rho u(|\rho|, h)
\end{equation}

 is harmonic in Int $\tilde \Omega_+$. Define ${\tilde\Omega_-}$ by $x \in \Omega \iff (-\rho(x), h (x)) \in \tilde \Omega_- . $ Let $\tilde \Omega$ be the union of $\tilde \Omega_+$ and $\tilde \Omega_-$. Then $\tilde \Omega$ is a  domain in $\mathbb{R}^2$, symmetric with respect to the line $\rho=0$. By the Schwarz reflection principle we see that \eqref{R4} defines an odd (with respect to $\rho$) harmonic function in $\tilde \Omega$.

\subsection{From $\mathbb{R}^{2k+3}$ to $\mathbb{R}^3$. \label{from 2k+3}} 
 Let $u=u(\rho,h)$ be an axially-symmetric harmonic function in $\mathbb{R}^{2k+3}$.
 Put
\begin{equation} \label{2k+3} v(\varphi, \rho,h)= \rho^k e^{ik\varphi}u(\rho,h), \end{equation}
 where $(\varphi, \rho,h)$ are cylindrical coordinates in $\mathbb{R}^3$. Then $v$ is  a harmonic (complex-valued) function in $\R^3$.  Indeed, 
 $$\Delta v = \frac{\partial^2 v}{\partial\rho^2} + \frac{1}{\rho}\frac{\partial v}{\partial\rho} +\frac{1}{\rho^2}\frac{\partial^2 v}{\partial\varphi^2} + \frac{\partial^2 v}{\partial h^2}=0+\rho^k e^{ik\varphi}(\frac{\partial^2 u}{\partial \rho^2}  + \frac{\partial^2 u}{\partial h ^2}  + \frac{2k+1}{\rho}\frac{\partial u}{\partial \rho} )=0.$$
 The last argument shows that $v$ is harmonic in $\R^3 \setminus \{ \rho =0 \}$. Note that $v$ is continuous up to the line $\{ \rho=0\}$, which is a removable singularity for bounded harmonic functions (see \cite{ABR},p.200). Thus $v$ is harmonic in $\R^3$.   
\section{ Proof of Theorem \ref{main}.}
\subsection{ Proof of the case $n=4$.}

 Fix $\varepsilon>0$: $R,H>\varepsilon$. Take any $x_0\in \R^{n-1}$  with $|x_0|<R-\varepsilon$. Consider any function $u$ from $\HH_M$. It is sufficient to show that there is $C=C(M,H,\varepsilon)$ such that $|u(x_0,h)|\leq C$ for any $h$: $|h|< H - \varepsilon$. Denote the set $ \{(x,y)\colon x \in \R^{n-1}, y \in \R, |x|<\varepsilon, |y|< H)\}$ by $P_\varepsilon$ and consider the function $\tilde u: P_\varepsilon \to \R$ defined by $\tilde u(x,y)= u(x-x_0,y)$. Note that $|\tilde u(x,y)| \leq M(|y|)$ on $P_\varepsilon$. 

Let us make an axial symmetrization step. Denote by $O(3)$ the group of orthogonal transformations in $\R^3$, let $dS$ be the Haar measure on $O(3)$. For any $g \in O(3)$  we use the notation $\tilde u_g$ for the function $\tilde u(gx,y)$. It is clear that $\tilde u_g$ is harmonic in $P_\varepsilon$, $\tilde u_g(0,y)=\tilde u(0,y)=u(x_0,y)$ and $|u_g(x,y)|\leq M(|y|)$  on $P_\varepsilon$.
 Put $w(x,y):= \int_{O(3)} u_g(x,y)dS(g)$, $(x,y)\in P_\varepsilon$, it is evident that $w$ also enjoys the  properties from the preceding sentence and $w=w(\rho,h)$ is axially-symmetric. We have reduced $4$-dimensional case to the following lemma.
\begin{lemma}
 Suppose $w=w(\rho,h)$ is an axially-symmetric harmonic function in the truncated cylinder $P_\varepsilon$ and $|w(x,y)| \leq M(|y|)$, then there is a constant $C=C(M,H,\varepsilon)$ such that $|w(0,y)|< C$ for any $y \in (-H+\varepsilon, H- \varepsilon)$.
\end{lemma}
 \textbf{Proof.}
   Put $v(\rho,h):=\rho w(|\rho|,h)$, by Section \ref{from R4} $v$ is harmonic in $(-\varepsilon,\varepsilon) \times (-H, H) $. Denote $\rho+ih$ by $\zeta$, and  $ \frac{\partial v}{\partial \rho} - i \frac{\partial v}{\partial h}$ by $f$, then $f$ is a holomorphic function in $(-\varepsilon,\varepsilon) \times (-H, H)$. Denote the set $(-\varepsilon/2,\varepsilon/2)\times (-H+\varepsilon/2, H - \varepsilon/2 )$ by $\tilde P_{\varepsilon/2}$. 

Take any $\zeta=(\rho,h) \in \tilde P_{\varepsilon/2}$ with $h \leq \varepsilon$ and consider a disk $B_{h/2}(\zeta):=\{z: |z-\zeta|< h/2 \}$. Since $|u(\rho,h)|\leq M(|h|)$ and $M$ is decreasing $\sup\{|v|(x)\colon x \in B_{h/2}(\zeta) \} \leq  M(h/2)$. Applying standard Cauchy's estimates of derivatives of harmonic functions we obtain $|\nabla v|(\zeta) \leq C_1 \frac{\sup\{|v|(x)\colon x \in B_{h/2}(\zeta) \}}{h/2} \leq C_2 \frac{M(h/2)}{h}$, by $C_1,C_2, C_3 $ we will denote absolute constants, whose value is  less than $100$. We note that $|f|=|\nabla v|$, hence $|f|(\zeta)\leq C_2 \frac{M(h/2)}{h}$.

 If $\zeta \in \tilde P_{\varepsilon/2}$ with $h \geq \varepsilon$, then $B_{\varepsilon/4}(\zeta) \subset (-\varepsilon,\varepsilon) \times (-H, H)$. Using in a similar way Cauchy's estimates we obtain $|f(\zeta)| \leq C_3 \frac{M(h/2)}{\varepsilon}$. We therefore have $|f(\zeta)| \leq \max(\frac{100}{\varepsilon}, \frac{100}{h})M(h/2)$ for any $\zeta \in \tilde P_{\varepsilon/2}$.    Denote $ \max(\frac{100}{\varepsilon}, \frac{100}{h})M(h/2)$ by $\tilde M(h)$. It follows from the inequality $ \log^+{a}+\log^+{b}+\log2 \geq \log^+(a+b)$ that $\int\limits_{-H}^{H}\log^{+}\log^{+}M(y)dy<+\infty$ implies  $\int\limits_{  -H+\varepsilon/2}^{ H-\varepsilon/2}\log^{+}\log^{+} \tilde M(y)dy<+\infty$.

Now, we are in a position to  apply Theorem \ref{Levinson} to the function $f$ holomorphic in $\tilde P_{\varepsilon/2}$ with the majorant $\tilde M$, that gives us  a positive constant $C=C(M,H,\varepsilon)$: $|f(0,h)|<C$ for $h \in (-H+\varepsilon, H-\varepsilon)$. Recall that $ v(\rho,h)= \rho \tilde  u(\rho,h)$, it yields $|u(0,h)|=|v_\rho(0,h)|\leq|f| \leq C(M,H,\varepsilon)$.

\begin{remark} Let $\tilde F(t)$ denote $\lambda_1( \{h\in (-H+\varepsilon/2, H -\varepsilon/2): \max(\frac{100}{\varepsilon}, \frac{100}{h})M(h/2) \geq \exp(t)\})$, then $C(M,H,\varepsilon)$ can be given explicitly in terms of $\tilde F$ in view of Remark \ref{estimateD}. Namely, if $\varepsilon/2 >  \frac{8}{\pi} \sum\limits_{i=-1}^{+\infty} \tilde F( 2^i C)$ for a positive constant $C$, then $u(x,y) \leq \exp(C)$ for all $(x,y)$ with $|x|\leq R-\varepsilon, |h| \leq H-\varepsilon$ .
\end{remark}
\begin{remark}
    The $4$-dimensional case of Theorem \ref{main} implies the $3$-dimensional one (as well as the $2$-dimensional) because we can always add a fake coordinate to $\R^3$. 
\end{remark}
\subsection{  Proof of the case $n\geq 5$.} We will consider only the case of odd $n=2k+3$. Now, we know that Theorem \ref{main} holds for $n=2,3,4$. We will prove the case of odd $n=2k+3$ reducing it to the case $n=3$ with the help of idea discussed in Section \ref{from 2k+3}. The case of even $n$ follows immediately.
    Like in the proof of $4$-dimensional case we can perform the axial-symmetrisation step and Theorem \ref{main} is reduced to the following lemma.
\begin{lemma}
Suppose $u=u(\rho,h)$ is an axially-symmetric harmonic function in a truncated cylinder $P_\varepsilon= \{(x \in \R^{n-1}, y \in \R, |x|<\varepsilon, |y|< H)\}$ such that $|u(x,y)| \leq M(|y|)$. Then there is a constant $\EuScript{C}=\EuScript{C}(n,M,H,\varepsilon)$ such that $|u(0,y)|< \EuScript{C}$ for $y \in (-H+\varepsilon, H- \varepsilon)$.
\end{lemma}
  Following Section \ref{from 2k+3} we consider a function $v$ defined by $v(\varphi,\rho,h)= Re( \rho^k e^{ik\varphi} u(\rho,h))$ on the set $\{\varphi \in [0,2\pi), \rho \in [0,\varepsilon), h \in (-H+\varepsilon,H+\varepsilon)  \}$, where $v$ is harmonic. With the help of a $3$-dimensional case of Theorem \ref{main} we can obtain $ |v(\varphi,\rho,h)|<C(M,H,\varepsilon/2)$ for $\varphi \in [0,2\pi), \rho \in [0,\varepsilon/2), h \in (-H+\varepsilon/2,H-\varepsilon/2)$. Then for any $h \in (-H+\varepsilon,H-\varepsilon)$  and the ball $B$ centered at the point $(0,0,h)$ with radius $\varepsilon/2$ we have $\sup\limits_{B} |v| \leq C(M,H,\varepsilon/2)$.  Applying standard estimates of the higher derivatives of harmonic functions we obtain $\frac{\partial^k}{\partial \rho^k}v \leq C(k) \frac{C(M,H,\varepsilon/2)}{(\varepsilon/2)^k}$ on the set $\{ \varphi\in [0,2\pi), \rho =0, h \in (-H+\varepsilon/2, H-\varepsilon/2) \}$ , where $C(k)$ is a constant depending only on dimension ($n=2k+3$). Take $\varphi=\rho=0$ and see that $\frac{\partial^k v}{\partial \rho^k}(0,0,h)= k! u(0,h)$. Thus $|u(0,h)| \leq C(k) \frac{C(M,H,\varepsilon/2)}{(\varepsilon/2)^k}$ for $h \in (-H+\varepsilon, H+\varepsilon)$.
\section{Application to the universal polynomial expansions of harmonic functions.}
  Consider the unit ball $\mathbb{B}:=B_1(0)$ in $\R^n$. Any function $h$ harmonic in $\mathbb{B}$ admits  power series expansion $h=\sum_{n=0}^{+\infty} h_n$, where $h_n$ is a homogeneous harmonic polynomial of degree $n$. It is said that $h$ belongs to the collection ${U}_H$, of harmonic functions in $B$ with universal homogeneous polynomial expansions, if for any compact set $K\subset \R^n \setminus \mathbb{B}$ with connected complement and any harmonic function $u$ in a neighbourhood of $K$, there is a subsequence $\{N_k\}$ of $\mathbb{N}$ such that $\sum^{N_k}_0 h_n \to u$ uniformly on $K$. This class of universal functions has been studied in \cite{M13}, \cite{GK14},  \cite{GT06}, \cite{BGNP}. The following statement improves Theorem $7$ from \cite{GK14} on the boundary behavior of functions from $U_H$ .
\begin{theorem} \label{universality}
 Let $\psi:[0,1) \to \R^+$ be an increasing function such that $\int_0^1 \log^+\log^+ \psi(t) dt < +\infty$.
 If $h=\sum_{n=0}^{+\infty} h_n$ enjoys $|h(x)| \leq \psi(|x|)$ on $B_r(\omega)\cap \mathbb{B}$ for some $\omega \in \partial \mathbb{B}$ and $r>0$, then $f \notin U_H$.
\end{theorem}
 We won't prove Theorem \ref{universality} here, because
 all necessary ingredients of the proof with one exception are given in \cite{GK14}, where Theorem \ref{universality} is proved under the stronger assumption $\int_0^1 \log^+ \psi(t) dt < +\infty$ in place of $\int_0^1 \log^+\log^+ \psi(t) dt < +\infty$. The only missing ingredient in \cite{GK14}, which allows to replace one $\log$ by $\log \log$, is the
 "harmonic" analog of the Levinson $\log\log$ theorem in higher dimensions (its version in a ball, which follows from Theorem \ref{main} with the help of Kelvin transform).

 \section{A question on one-sided estimates.}
  Suppose that $z_0$ is a point in a square $Q=(-1,1)\times (-1,1)$ and $M$ is a positive (decreasing and regular) function on $(0,1)$. Under what assumptions on $M$  the family $F^+_M$ of all functions $f$ holomorphic in $Q$, and satisfying $Im (f(z))\leq M(|Im(z)|)$, $f(z_0)= 0$ is normal in $Q$?

%

\section*{Acknowledgments} This research was supported by the Chebyshev Laboratory (Department of Mathematics
and Mechanics, St. Petersburg State University) under the RF Government grant
11.G34.31.0026, and by JSC "Gazprom Neft". We are grateful to D. Khavinson for explaining the application of the "harmonic" analog of the Levinson $\log \log$ theorem to the question on the boundary behavior  of the universal power expansions of harmonic functions. 

\end{document}